\documentclass[11pt]{article}
\usepackage[english]{babel}
\usepackage{amsmath,amsthm}
\usepackage{amsfonts}
\usepackage{mathrsfs}
\usepackage{color}
\usepackage{bbm}
\usepackage{enumerate}
\usepackage{amsrefs}
\usepackage{mathtools}
\usepackage{hyperref}

\newtheorem{theorem}{Theorem}[section]

\newtheorem{lemma}[theorem]{Lemma}
\newtheorem{proposition}[theorem]{Proposition}
\newtheorem{assumption}{Assumption}
\theoremstyle{definition}
\newtheorem{definition}[theorem]{Definition}
\theoremstyle{remark}
\newtheorem{remark}[theorem]{Remark}

\numberwithin{equation}{section}
\begin{document}
\def\Pro{{\mathbb{P}}}
\def\E{{\mathbb{E}}}
\def\e{{\varepsilon}}
\def\veps{{\varepsilon}}
\def\ds{{\displaystyle}}
\def\nat{{\mathbb{N}}}
\def\Dom{{\textnormal{Dom}}}
\def\dist{{\textnormal{dist}}}
\def\R{{\mathbb{R}}}
\def\O{{\mathcal{O}}}
\def\T{{\mathcal{T}}}
\def\Tr{{\textnormal{Tr}}}
\def\sgn{{\textnormal{sign}}}
\def\I{{\mathcal{I}}}
\def\A{{\mathcal{A}}}
\def\H{{\mathcal{H}}}
\def\S{{\mathcal{S}}}

\title{Global solutions for the stochastic reaction-diffusion equation with super-linear multiplicative noise and strong dissipativity}%
\author{M. Salins\\ Boston University \\ msalins@bu.edu}
\maketitle

\begin{abstract}
  A condition is identified that implies that solutions to the stochastic reaction-diffusion equation
$\frac{\partial u}{\partial t} = \mathcal{A} u + f(u) + \sigma(u) \dot{W}$
on a bounded spatial domain never explode. We consider the case where $\sigma$ grows polynomially and $f$ is polynomially dissipative, meaning that $f$ strongly forces solutions toward finite values. This result demonstrates the role that the deterministic forcing term $f$ plays in preventing explosion.

\end{abstract}

%
%

\section{Introduction} \label{S:intro}
  We investigate the role that dissipative forcing plays in preventing explosion of solutions to the stochastic reaction-diffusion equation (SRDE) defined on an open bounded domain $D \subset \mathbb{R}^d$ with appropriately smooth boundary
  \begin{equation} \label{eq:SRDE}
    \begin{cases}
      \displaystyle
      \frac{\partial u}{\partial t}(t,x) = \mathcal{A} u(t,x)  + f(u(t,x)) + \sigma(u(t,x))\dot{W}(t,x)\\
      u(t,x) = 0 , \ \ \ x \in \partial D,\\
      u(0,x) = u_0(x).
    \end{cases}
  \end{equation}
  In the above equation, $\mathcal{A}$ is a second-order linear operator and $\dot{W}$ is a Gaussian noise. The function $f(u(t,x))$ models a state-dependent external force and $\sigma(u(t,x))\dot{W}$ models a state-dependent stochastic forcing.

  The presence of stochastic forcing can cause solutions to become arbitrarily large with positive probability and may even cause solutions to explode in finite time. Mueller and collaborators \cite{m-1991,m-1998,ms-1993,m-2000} investigated explosion for solutions to \eqref{eq:SRDE} when $\mathcal{A} = \Delta$, $f(u)\equiv 0$ ,  $\dot{W}$ is a space-time white noise and the spatial dimension is $d=1$.
  Solutions can explode in finite time if $|\sigma(u)|> c|u|^\gamma$ for some $c>0$ and $\gamma> \frac{3}{2}$. Furthermore, this $\frac{3}{2}$ power is critical in the sense that solutions never explode if $|\sigma(u)| \leq C(1 + |u|^\gamma)$ for some $C>0$ and $\gamma<\frac{3}{2}$. For other investigations of explosion of similar equations see \cite{chow-2009,chow-2011,fln-2019,b-2018}. In this paper, we demonstrate that $\gamma$ can be arbitrarily large and solutions will never explode as long as the strong stochastic forcing is compensated by an appropriately strong dissipative forcing $f$.

  Early investigations of these reaction-diffusion equations proved that there exists a unique solution to \eqref{eq:SRDE} whenever $f$ and $\sigma$ are both globally Lipschitz continuous with at most linear growth \cite{k-1992,s-1992,d-1999,ss-2000,w-1986,dpz-book,wr-book,pz-2000,dkmnx-book}. Linearly growing $f$, however, will not be strong enough to prevent the expansive effects of a superlinearly growing $\sigma$.
  The existence of global solutions for \eqref{eq:SRDE} when $\sigma$ is locally Lipschitz continuous with linear growth and $f$ is strongly dissipative with polynomial growth was established by Cerrai \cite{c-2003}. The restriction to polynomial growing $f$ is common in the literature \cite{mz-1999,i-1987,bp-1999,c-2011}, but Da Prato and R\"ockner \cite{dr-2002} and Marinelli and R\"ockner \cite{mr-2010} proved that the polynomial growth restriction can be relaxed, and that existence and uniqueness of solutions to \eqref{eq:SRDE} is implied by a monotonicity condition on $f$. See also \cite{m-2018,cdf-2013,mr-2010,grt-2020}. These monotonicity conditions allow for strongly dissipative forcing terms.

  In this paper, we assume $\sigma$ grows polynomially like $|\sigma(u)|\leq K_2(1 + |u|^\gamma)$ for some $\gamma>1$ and we assume that there exist $K_1>0$, $c_0>0$, and $\beta>1$ such that $f(u)\sgn(u) \leq -K_1 |u|^\beta$ for $|u|>c_0$. In particular, this means that $f(u)$ pushes solutions away from $\pm \infty$ when $|u|$ is large.  This strong dissipation can counteract the expansive effects of the stochastic forcing, preventing explosion.
%
%

  Before introducing the exact relationship between $\beta$ and $\gamma$ that implies that the solutions to the reaction-diffusion are global, consider a related problem for a finite dimensional stochastic differential equation defined on $\mathbb{R}^d$. Let $B(t)$ be a $d$-dimensional Wiener process, let $\beta, \gamma>0$, and let $X(t)$ solve
  \begin{equation} \label{eq:SDE}
    dX(t) = -|X(t)|^{\beta -1} X(t) dt + (1 + |X(t)|)^\gamma dB(t), \ \ \ X(0) = x \in \mathbb{R}^d.
  \end{equation}
  For $R>0$, let $\tau_R = \inf\{t>0: |X(t)| >R\}$. The coefficients are all locally Lipschitz continuous, so $X(t \wedge \tau_R)$ is a well-defined stochastic process.
  By Ito formula, for any $R>0$, $T>0$,
  \begin{equation} \label{eq:Ito}
    \E|X(T\wedge \tau_R)|^2 \leq |x|^2 +  \E \int_0^{T \wedge \tau_R} \left(- 2 |X(t \wedge \tau_R)|^{\beta + 1} + (1  + |X(t \wedge \tau_R)|)^{2\gamma}\right)dt.
  \end{equation}
  {\color{black}
  If $\gamma\in [0, 1]$ then the multiplicative noise coefficent is sublinear and a Gr\"onwall argument proves that for some $C>0$ independent of $R$,
  \begin{equation}
    \E|X(T\wedge \tau_R)|^2 \leq (|x|^2 + CT) e^{CT}.
  \end{equation}
  Letting $R \to \infty$ we see that  solutions cannot explode. The value of $\beta \geq 0$ does not affect explosion $\gamma \in [0,1]$.

  We are particularly interested in the case where $\gamma>1$ so that the multiplicative noise term is superlinear. A Gr\"onwall argument cannot be used in this setting, but if $\beta$ is large enough so that}
  \begin{equation} \label{eq:Ito-cond}
   \gamma< \frac{\beta + 1}{2}
  \end{equation}
   then the integrand in \eqref{eq:Ito} is uniformly bounded by a constant independent of $R$. Therefore, for any $R>0$, $T>0$,
   \begin{equation}
     \E|X(T \wedge \tau_R)|^2 \leq |x|^2 +  CT.
   \end{equation}
   We can prove that solutions to the finite dimensional SDE never explode by taking the limit as $R \to \infty$ on the left.

  In the case of SRDEs \eqref{eq:SRDE}, however, such an Ito formula argument will not work unless the stochastic forcing has a trace-class covariance. We will not assume trace-class covariance in general, although we will prove that the conditions of our main result are almost the same as the Ito formula condition \eqref{eq:Ito-cond} in the trace-class setting.

{\color{black}
  If $\sigma$ grows at most linearly ($\gamma \in [0,1]$) and $f$ is dissipative, then solutions to \eqref{eq:SRDE} cannot explode in finite time (see, for example, \cite{c-2003}). When $\sigma$ grows polynomially, however, solutions can explode unless the superlinear stochastic force is compensated by an appropriately strong dissipative force \cite{ms-1993,m-2000}. For this reason, we only study the case where both $\gamma>1$ and $\beta>1$.
}

  The main result of this paper (Theorem \ref{thm:global-existence}) proves that mild solutions to the stochastic reaction diffusion-equation never explode if
  \begin{equation} \label{eq:gamma-beta-intro}
    \gamma < 1 + \frac{(1-\eta)(\beta-1)}{2}
  \end{equation}
  where $\eta \in [0,1)$ is a constant that describes the balance between the eigenvalues of the elliptic operator $\mathcal{A}$ and of the noise $\dot{W}$ (see Assumption \ref{assum:W}, below). Note that when $\eta =0$, \eqref{eq:Ito-cond} and \eqref{eq:gamma-beta-intro} coincide.

  {\color{black}
  A mild solution to \eqref{eq:SRDE} solves the integral equation
  \begin{equation}
    u(t) = S(t)u(0) + \int_0^t S(t-s)f(u(s))ds  + \int_0^t S(t-s)\sigma(u(s))dW(s)
  \end{equation}
  where $S(t)$ is the semigroup generated by the elliptic operator $\mathcal{A}$. The spatial variable is suppressed in the above equation.

  Instead of using Ito formula, we will take advantage of the fact that super-linear dissipativity terms like $f(u) = -|u|^{\beta-1} u$ lead to estimates on the solutions that are independent of the initial data. Consider, for example, the deterministic ordinary differential equation
  \begin{equation} \label{eq:ODE}
    \frac{d\phi}{dt} = -|\phi(t)|^{\beta-1} \phi(t).
  \end{equation}
  The solution to this equation is
  \begin{equation}
    \phi(t) =  \pm \left( |\phi(0)|^{-(\beta -1)} + (\beta -1) t \right)^{- \frac{1}{\beta -1}} .
  \end{equation}
  Unlike in the linear ($\beta=1$) case, when $\beta>1$ we can obtain bounds that are independent of initial data. In particular, for any $t>0$,
  \begin{equation}
    |\phi(t)| \leq C t^{-\frac{1}{\beta-1}} \text{ for any initial data}.
  \end{equation}

  A similar bound will hold for the mild solution. In particular, Lemma \ref{lem:decay} below proves that if the stochastic convolution term \begin{equation} \label{eq:intro-stoch-conv}
    Z(t): = \int_0^t S(t-s)\sigma(u(s))dW(s)
  \end{equation}
   is relatively small compared to $u$ in the sense that $|Z(t)|_{L^\infty(D)} \leq \frac{1}{3} |u(t)|_{L^\infty(D)}$ for $t \in [0,T]$, then
  \begin{equation} \label{eq:intro-decay}
    |u(t)|_{L^\infty(D)} \leq \frac{3}{2}\left( |u(0)|_{L^\infty(D)}^{-(\beta-1)} + \frac{K_1}{2^\beta(\beta-1)} t \right)^{-\frac{1}{\beta -1}}
  \end{equation}

  On the other hand, estimates from \cite{c-2003,c-2009} prove that the $p$th moments of the supremum of the stochastic convolution \eqref{eq:intro-stoch-conv} satisfy bounds like
  \begin{equation} \label{eq:intro-factor}
    \E \sup_{t \in [0,T]} \sup_{x \in D} |Z(t,x)|^p \leq C_p E\int_0^T  \E\left(\int_0^t (t-s)^{-2 \alpha - \eta} |\sigma(u(s))|_{L^\infty(D)}^2ds\right)^{\frac{p}{2}}dt
  \end{equation}
  where $\eta \in (0,1)$ is the constant in \eqref{eq:gamma-beta-intro} and $\alpha \in \left(0, \frac{1-\eta}{2}\right)$. Because $\sigma(u)$ grows like $|u|^\gamma$ and $u(t)$ decays like \eqref{eq:intro-decay} due to the strong dissipativity of $f$, the inner integral can be bounded by
  \begin{equation} \label{eq:intro-linear}
    C |u(0)|_{L^\infty(D)}^2 \int_0^t (t-s)^{-\eta-2\alpha}  s^{-\frac{(\gamma-1)}{(\beta-1)}}ds.
  \end{equation}
  The Beta function $\int_0^t (t-s)^{-\eta-2\alpha}  s^{-\frac{(\gamma-1)}{(\beta-1)}}ds$ is uniformly bounded for $t \in [0,1]$ if and only if
  \begin{equation}
    \eta + 2\alpha + \frac{(\gamma-1)}{(\beta-1)} \leq 1,
  \end{equation}
  which is equivalent to condition \eqref{eq:gamma-beta-intro}. Despite the fact that $\sigma$ is superlinear, the estimates \eqref{eq:intro-factor}--\eqref{eq:intro-linear} show that the size of the stochastic convolution only depends linearly on the initial value of $u(0)$.

  To make these ideas rigorous, we introduce a sequence of stopping times that keep track of when the spatial $L^\infty$ norm triples or falls to one-third of its previous value. Under condition \ref{eq:gamma-beta-intro}, we prove using the ideas above that the $L^\infty$ norm cannot triple in a short amount of time. This prevents the mild solutions from exploding.
}

  In Section \ref{S:assum} we introduce our assumptions and state the main result. In Section \ref{S:estimates} we recall important estimates from \cite{c-2003,c-2009} and we prove new estimates on solutions that are uniform with respect to initial data. In Section \ref{S:proof-of-main}, we prove the main result. We conclude with a discussion in Section \ref{S:comparison} about the relationship between the main result of this paper and Mueller's explosion results \cite{m-2000}.

\section{Notation, assumptions and main result} \label{S:assum}
\subsection{Notation} \label{S:notation}
Let $D \subset \mathbb{R}^d$ be an open, bounded domain. For $p\in [1,+\infty)$, define $L^p(D)$ to be the Banach space of functions $v: D \to \mathbb{R}$ such that the norm
\begin{equation}
  |v|_{L^p(D)} := \left(\int_D |v(x)|^pdx\right)^{\frac{1}{p}}
\end{equation}
is finite.
When $p=+\infty$, the $L^\infty(D)$ norm is
\begin{equation}
  |v|_{L^\infty(D)} : = \sup_{x \in D} |v(x)|.
\end{equation}

Define $C_0(\bar D)$ to be the subset of $L^\infty(D)$ of continuous functions $v: \bar D \to \mathbb{R}$ such that $v(x) = 0$ for $x \in \partial D$.
Define $C_0([0,T]\times \bar D)$ to be the set of continuous functions $v: [0,T] \times \bar D \to \mathbb{R}$ such that $v(t,x) = 0$ for $x \in \partial D$, endowed with the supremum norm
\begin{equation}
  |v|_{C_0([0,T]\times \bar D)} := \sup_{t \in [0,T]} \sup_{x \in \bar D} |v(t,x)|.
\end{equation}

\subsection{Assumptions}
  We make the following assumptions about the differential operator $\mathcal{A}$, the noise $\dot{W}$, and the deterministic and stochastic forcing terms $f$ and $\sigma$ in \eqref{eq:SRDE}.

  \begin{assumption} \label{assum:A}
    $\mathcal{A}$ is a second-order elliptic differential operator
    \begin{equation}
      \mathcal{A} \phi(x) = \sum_{i=1}^d \sum_{j=1}^d a_{ij}(x) \frac{\partial^2 \phi}{\partial x_i\partial x_j} + \sum_{i=1}^d b_i(x) \frac{\partial \phi}{\partial x_i}
    \end{equation}
    where $a_{ij}$ are continuously differentiable on $\bar{D}$ and $b_i$ are continuous on $D$.

    As observed in \cite{c-2003}, we can assume without loss of generality that $\mathcal{A}$ is self-adjoint. We make this assumption throughout the rest of the paper.

     Let $A$ be the realization of $\mathcal{A}$ in $L^2(D)$ with the Dirichlet boundary conditions. There exists a sequence of eigenvalues $0\leq \alpha_1 \leq \alpha_2 \leq ...$ and eigenfunctions $e_k \in L^2(D) \cap C_0(D)$ such that \cite[Chapter 6.5]{evans-book}
    \begin{equation}
      A e_k = -\alpha_k e_k \text{ and } |e_k|_{L^2(D)} = 1.
    \end{equation}
  \end{assumption}

  \begin{assumption} \label{assum:W}
    There exists a sequence of numbers $\lambda_j \geq 0$ and a sequence of i.i.d.  one-dimensional Brownian motions $\{B_j(t)\}_{j \in \mathbb{N}}$ such that formally
    \begin{equation}
      \dot W(t,x) = \sum_{k=1}^\infty \lambda_j e_j(x) dB_j(t).
    \end{equation}
    Furthermore, there exist exponents $\theta>0$ and $\rho \in [2,+\infty)$ such that
    \begin{equation} \label{eq:lambda-sum-assum}
      \begin{cases}
        \left(\sum_{j=1}^\infty \lambda_j^\rho |e_j|_{L^\infty(D)}^2 \right)^{\frac{2}{\rho}}< +\infty, & \text{ if } \rho \in [2,+\infty), \\
        \sup_j \lambda_j < +\infty & \text{ if } \rho=+\infty,
      \end{cases}
    \end{equation}
    \begin{equation} \label{eq:alpha-sum}
      \sum_{k=1}^\infty \alpha_k^{-\theta} |e_k|_{L^\infty(D)}^2 < +\infty,
    \end{equation}
    and
    \begin{equation} \label{eq:eta-def}
      \eta: = \frac{\theta (\rho-2)}{\rho} <1.
    \end{equation}
  \end{assumption}

  The constant $\eta<1$ defined in \eqref{eq:eta-def} is central to our analysis and shows up as a condition in our main theorem. The trace-class noise case corresponds to $\rho=2$ implying that $\eta=0$. In the case of space-time white noise on a one-dimensional spatial interval, we can take $\rho=+\infty$ and $\eta = \theta$ can be any number larger than $\frac{1}{2}$.

  \begin{assumption} \label{assum:f-sig}
    $f: \mathbb{R} \to \mathbb{R}$ and $\sigma: \mathbb{R} \to \mathbb{R}$ are continuous functions. There exist powers $\beta, \gamma>1$ and constants $c_0>0$ and $K_1, K_2>0$ such that
    \begin{equation} \label{eq:f-assum}
      f(u)\sgn(u) \leq -K_1 |u|^\beta \text{ for } |u|>c_0,
    \end{equation}
    \begin{equation} \label{eq:sig-assum}
      |\sigma(u)| \leq K_2(1 + |u|^\gamma) \text{ for } u \in \mathbb{R}^d
    \end{equation}
    and
    \begin{equation} \label{eq:beta-gamma-condition}
      \gamma < 1 + \frac{(1-\eta)(\beta-1)}{2}.
    \end{equation}
  \end{assumption}

  Notice that in the trace-class noise situation where $\eta=0$, condition \eqref{eq:beta-gamma-condition} matches condition \eqref{eq:Ito-cond} from the SDE case.

  \begin{assumption} \label{assum:init-data}
    The initial data $u_0 \in C_0(\bar{D})$.
  \end{assumption}

\subsection{Main result}
  Under Assumption \ref{assum:A}, the realization $A$ of $\mathcal{A}$ in $L^2(D)$ generates a $C_0$ semigroup $S(t)$.

  \begin{definition}
    A $C_0(\bar D)$-valued process $u(t)$ is \textit{local mild solution} to \eqref{eq:SRDE} \textcolor{black}{if }
    \begin{equation} \label{eq:mild}
      u(t) = S(t)u_0 + \int_0^t S(t-s)f(u(s))ds  + \int_0^t S(t-s)\sigma(u(s))dW(s)
    \end{equation}
    \textcolor{black}{for all $t \in [0,T_n]$ for any $n$ where $T_n$ is the stopping time
    \begin{equation}
      T_n: = \inf\{t>0: |u(t)|_{L^\infty(D)} \geq n\}.
    \end{equation}}
  \end{definition}

  \begin{definition}
    A mild solution $u$ is \textit{global} if $u(t)$ solves \eqref{eq:mild} for all $t>0$, with probability one. In other words, a solution is global if it never explodes.
  \end{definition}

  Now we present our main theorem.

  \begin{theorem} \label{thm:global-existence}
    Under Assumptions \ref{assum:A}--\ref{assum:init-data}, any local mild solution to \eqref{eq:SRDE} is a global solution.
  \end{theorem}
  The proof of Theorem \ref{thm:global-existence} is in Section \ref{S:proof-of-main}.
{\color{black}
\begin{remark}
  Notice that we make no claims about existence or uniqueness of mild solutions. Instead, we claim that if a local mild solution exists, then it cannot explode in finite time. If we add the reasonable assumption that $f$ and $\sigma$ are both locally Lipschitz continuous then a standard localization argument proves that there exists a unique local, mild solution (see for example \cite[Proof of Theorem 5.3]{c-2003}). We do not include this assumption to emphasize that it is really the tail behaviors \eqref{eq:f-assum}, \eqref{eq:sig-assum}, and the condition \eqref{eq:beta-gamma-condition} that prevent explosion.
\end{remark}
}
\section{Estimates} \label{S:estimates}
\subsection{Moment bounds of the supremum of the stochastic convolution}
By the factorization method of Da Prato and Zabczyk \cite[Chapter 5.3.1]{dpz-book} (see also \cite{c-2003}), a stochastic integral
\begin{equation} \label{eq:Z-stoch-conv}
  Z(t)= \int_0^t S(t-s)\sigma(u(s))dW(s)
\end{equation}
can be written as
\begin{equation} \label{eq:factor-Z}
  Z(t) = \frac{\sin(\pi \alpha)}{\pi} \int_0^t (t-s)^{\alpha-1} S(t-s) Z_\alpha(s)ds
\end{equation}
where $\alpha \in (0,1)$ and
\begin{equation}
  Z_\alpha(t) = \int_0^t (t-s)^{-\alpha} S(t-s)\sigma(u(s))dW(s).
\end{equation}

{\color{black}
Now let $\tau$ be a stopping time with respect to the natural filtration of $W(t)$. Using the factorization formula, $Z(t \wedge \tau)$ can be written as
\begin{equation} \label{eq:factor}
  Z(t \wedge \tau) = \frac{\sin(\pi \alpha)}{\pi} \int_0^{t \wedge \tau} (t \wedge \tau -s)^{\alpha -1} S(t \wedge \tau -s) Z_\alpha(s)ds.
\end{equation}
This is also equal to
\begin{equation} \label{eq:factor-tilde}
  Z(t \wedge \tau) = \frac{\sin(\pi \alpha)}{\pi}\int_0^{t \wedge \tau} (t \wedge \tau -s)^{\alpha -1} S(t \wedge \tau -s) \tilde Z_{\alpha}(s)ds
\end{equation}
where
\begin{equation}
  \tilde Z_\alpha(t) = \int_0^t (t-s)^{-\alpha} S(t-s) \sigma(u(s ))\mathbbm{1}_{\{s \leq \tau\}}dW(s).
\end{equation}
Expressions \eqref{eq:factor} and \eqref{eq:factor-tilde} are equal because $Z_\alpha(t) = \tilde Z_\alpha(t)$ for all $t \leq \tau$.

We prove the following two propositions in the appendix.
\begin{proposition} \label{prop:Z-alpha}
  Let $\alpha \in \left( 0, \frac{1-\eta}{2}\right)$ and  $p\geq 2$. For any $t>0$,
  \begin{equation} \label{eq:BDG-stoch-conv}
  \E |\tilde{Z}_\alpha(t)|_{L^p(D)}^p \leq C_{\alpha,p} \E \left(\int_0^t (t-s)^{-\eta - 2\alpha} |\sigma(u(s))|_{L^\infty(D)}^2 \mathbbm{1}_{\{s \leq \tau\}}ds  \right)^{\frac{p}{2}}
  \end{equation}
\end{proposition}

\begin{proposition} \label{prop:Z-sup-bound}
  Let $\tau$ be a stopping time with respect to the natural filtration of $W(t)$. If $\E\sup_{t \in [0,\tau]}|u(t)|^p_{L^\infty(D)}<+\infty$, then $(t,x) \mapsto Z(t \wedge \tau, x)$ is almost surely continuous. Furthermore, for any $\alpha \in \left( 0, \frac{1-\eta}{2} \right)$, $\zeta \in (0,2\alpha)$ and  $p> \max\left\{ \frac{d}{\zeta}, \frac{1}{\alpha-\frac{\zeta}{2}}\right\}$,
  \begin{equation} \label{eq:stoch-conv-sup-norm}
    \E \sup_{s \in [0,t]} \sup_{x \in D} |Z(s,x)|^p \leq C_{\alpha,\zeta,p} t^{p(\alpha-\frac{\zeta}{2}) -1} \int_0^t \E|\tilde{Z}_\alpha(s)|_{L^p(D)}^pds.
  \end{equation}
\end{proposition}
The proofs of Propositions \ref{prop:Z-alpha} and \ref{prop:Z-sup-bound} are very similar to the proofs of \cite[Theorem 4.2]{c-2003} and \cite[Lemma 4.1]{c-2009}, although the inclusion of the stopping time is slightly different. For completeness, we include the proofs in the appendix.
}
%
%

\subsection{Uniform bounds}
Before directly analyzing the properties of the mild solution \eqref{eq:mild}, we consider an associated deterministic problem.
Let $z \in C_0([0,T]\times \bar D)$ be a continuous function of space and time and  assume that $u(t)$ solves the integral equation (with the spatial variable suppressed )
\begin{equation} \label{eq:determ-int-eq}
  u(t) = S(t)u_0 + \int_0^t S(t-s)f(u(s))ds + z(t)
\end{equation}
If $u(t)$ is a mild solution to \eqref{eq:SRDE}, then $u(t)$ satisfies \eqref{eq:determ-int-eq} where $z(t)$ is replaced with the stochastic convolution \eqref{eq:Z-stoch-conv}.

\begin{lemma} \label{lem:decay}
  Assume that $z,u \in {C_0([0,T]\times D)}$  solve \eqref{eq:determ-int-eq}. If
  \begin{equation}
    |z(t)|_{L^\infty(D)} \leq \frac{1}{3} |u(t)|_{L^\infty(D)} \text{ and } 3c_0 < |u(t)|_{L^\infty(D)} \text{ for all $t \in [0,T]$,}
  \end{equation}
  then for all $t \in [0,T]$,
  \begin{equation}
    |u(t)|_{L^\infty(D)} \leq \frac{3}{2}\left( |u(0)|_{L^\infty(D)}^{-(\beta-1)} + \frac{K_1}{2^\beta(\beta-1)} t \right)^{-\frac{1}{\beta -1}}.
  \end{equation}
\end{lemma}

\begin{proof}
  Assume that $z(t)$ and $u(t)$ are as described above. Let $v(t) = u(t) - z(t) = S(t)u_0 + \int_0^t S(t-s)f(v(s)+z(s))ds$.
Then $v$ is weakly differentiable and weakly solves the partial differential equation
\begin{equation} \label{eq:v-pde}
  \frac{\partial v}{\partial t}(t,x) = \mathcal{A} v(t,x) + f(v(t,x) + z(t,x)).
\end{equation}
By a standard Yosida approximation argument (see Proposition 6.2.2 of \cite{cerrai-book} or Theorem 7.7 of \cite{dpz-book}), we can assume without loss of generality that $v$ is a strong solution of \eqref{eq:v-pde}.
By Proposition D.4 in the appendix of \cite{dpz-book}, $t \mapsto |v(t)|_{L^\infty(D)}$ is left-differentiable and
\begin{equation} \label{eq:sup-ode}
  \frac{d^-}{dt} |v(t)|_{L^\infty(D)} \leq \mathcal{A} v(t,x_t) \sgn(v(t,x_t)) + f(v(t,x_t) + z(t,x_t))\sgn(v(t,x_t))
\end{equation}
where $x_t \in D$ is a maximizer satisfying
\begin{equation}
  |v(t)|_{L^\infty(D)} = |v(t,x_t)| = v(t,x_t)\sgn(v(t,x_t)).
\end{equation}

Because $\mathcal{A}$ is elliptic, by the convexity of a function at its maximum or minimum,
\begin{equation} \label{eq:A-dissip}
  \mathcal{A}v(t,x_t) \sgn(v(t,x_t)) \leq 0.
\end{equation}
By the triangle inequality,
\begin{equation}
  |v(t)|_{L^\infty(D)} \geq |u(t)|_{L^\infty(D)} - |z(t)|_{L^\infty(D)}.
\end{equation}
By assumption, for $t \in [0,T]$, $3|z(t)|_{L^\infty(D)} \leq |u(t)|_{L^\infty(D)}$, and therefore,
\begin{equation} \label{eq:v-double-z}
  |v(t)|_{L^\infty(D)} \geq 2 |z(t)|_{L^\infty(D)}.
\end{equation}
From these estimates it follows that
\begin{equation} \label{eq:v-lower-bound}
  |v(t,x_t) + z(t,x_t)| \geq |v(t)|_{L^\infty(D)} - |z(t)|_{L^\infty(D)} \geq \frac{1}{2}|v(t)|_{L^\infty(D)}.
\end{equation}
And by the assumption that $|u(t)|_{L^\infty(D)}>3c_0$,
we see that
\begin{equation} \label{eq:v-lower-bound-2}
  |v(t)|_{L^\infty(D)} \geq |u(t)|_{L^\infty(D)} - |z(t)|_{L^\infty(D)} \geq \frac{2}{3} |u(t)|_{L^\infty(D)} \geq 2 c_0.
\end{equation}
Therefore, by \eqref{eq:v-lower-bound} for $t \in [0,T]$,
\begin{equation}
  |v(t,x_t) + z(t,x_t)| \geq \frac{1}{2}|v(t)|_{L^\infty(D)} \geq c_0.
\end{equation}
By \eqref{eq:v-double-z}, $\sgn(v(t,x_t)) = \sgn(v(t,x_t) + z(t,x_t))$. Therefore,
by \eqref{eq:f-assum}, \eqref{eq:sup-ode},  \eqref{eq:A-dissip}, and \eqref{eq:v-lower-bound},  for $t \in [0,T]$,
\begin{equation}
  \frac{d^-}{dt}|v(t)|_{L^\infty(D)} \leq -\frac{K_1}{2^\beta} |v(t)|_{L^\infty(D)}^\beta.
\end{equation}

Let $F(u) = - (\beta -1)u^{-(\beta -1)}$ so that $F'(u) = u^{-\beta}$. Then because $F$ is differentiable and increasing,
\begin{equation}
  \frac{d^-}{dt} F(|v(t)|_{L^\infty(D)}) \leq -\frac{K_1}{2^\beta}
\end{equation}
and therefore,
\begin{equation}
  F(|v(t)|_{L^\infty(D)}) \leq F(|u(0)|_{L^\infty(D)}) -\frac{K_1t}{2^\beta}
\end{equation}
and
\begin{equation}
  |v(t)|_{L^\infty(D)} \leq \left(|u(0)|_{L^\infty(D)}^{-(\beta-1)} + \frac{K_1 t}{2^\beta(\beta-1)} \right)^{-\frac{1}{\beta-1}}.
\end{equation}
Finally, we use the estimate \eqref{eq:v-lower-bound-2} to see that
\begin{equation}
  |u(t)|_{L^\infty(D)} \leq \frac{3}{2}|v(t)|_{L^\infty(D)} \leq \frac{3}{2}\left(|u(0)|_{L^\infty(D)}^{-(\beta-1)} + \frac{K_1 t}{2^\beta(\beta-1)} \right)^{-\frac{1}{\beta-1}}.
\end{equation}

\end{proof}

\section{Proof of Theorem \ref{thm:global-existence}} \label{S:proof-of-main}
To prove that mild solutions to \eqref{eq:SRDE} are global in time, we build a sequence of stopping times. Let $u(t)$ be a mild solution to \eqref{eq:SRDE}. Let $c_0$ be the constant from Assumption \ref{assum:f-sig} and let
\begin{align}\label{eq:tau-def}
    &\tau_0 = \inf\{t\geq 0: |u(t)|_{L^\infty(D)} = 3^n c_0 \text{ for some } n \in \{1,2,3,4,...\}\} \nonumber\\
    &\tau_{k+1} =
       \begin{cases}
         \inf\{t \geq \tau_k: |u(t)|_{L^\infty(D)} \geq 3^2 c_0 \} & \text{ if } |u(\tau_k)|_{L^\infty(D)} = 3c_0\\
         \inf\{t \geq \tau_k: |u(t)|_{L^\infty(D)} \geq 3|u(\tau_k)|_{L^\infty(D)} \\
         \hspace{1.3cm} \text{ or } |u(t)|_{L^\infty(D)} \leq \frac{1}{3}|u(\tau_k)|_{L^\infty(D)}\} & \text{ if } |u(\tau_k)|_{L^\infty(D)} \geq 3^2 c_0.
       \end{cases}
\end{align}

If mild solutions explode in finite time then $|u(\tau_k)|_{L^\infty(D)}$ will diverge to $+\infty$ while $\sup_k \tau_k< +\infty$. We will demonstrate that this cannot happen.

\begin{lemma}  \label{lem:cond-prob-grow-quick}
  There exist constants $C>0$ and $q>1$, independent of  $n$, $k$, and $\e>0$, such that for any $\e>0$, \textcolor{black}{any $k \in \mathbb{N}$,} and any $n \in \{2,3,4,5,6,...\}$,
  {\color{black}
  \begin{equation}
    \Pro \left( |u(\tau_{k+1})|_{L^\infty(D)} = 3 |u(\tau_k)|_{L^\infty(D)} \text{ and } \tau_{k+1} - \tau_k<\e \ \Big| \  |u(\tau_k)| = 3^n c_0\right) \leq C \e^{q}.
  \end{equation}
}
\end{lemma}

\begin{proof}
  Because $u(t)$ is a local mild solution to \eqref{eq:SRDE}, \textcolor{black}{for $t\in [0,\tau_{k+1} - \tau_k]$ and $k \in \mathbb{N}$, $u(t)$ solves
  \begin{align}
    u(t + \tau_k) &= S(t)u(\tau_k) + \int_{\tau_k}^{\tau_k + t} S(t + \tau_k -s) f(u(s))ds \nonumber\\
     &\hspace{2cm}+\int_{\tau_k}^{\tau_k + t} S(t + \tau_k -s) \sigma(u(s))dW(s) \nonumber\\
     &= S(t)u(\tau_k) + \int_{0}^{t} S(t  -s) f(u(s+ \tau_k))ds \nonumber\\
     &\hspace{2cm}+\int_{0}^{t} S(t  -s) \sigma(u(s + \tau_k))dW(s + \tau_k).
  \end{align}
  Because $dW(t)$ is white in time, $s \mapsto dW(s + \tau_k)$ conditioned on $\mathcal{F}_{\tau_k}$  has the same distribution as $s \mapsto dW(s)$. Without loss of generality, we can assume that $k=0$ and it suffices to prove that
  \begin{equation}
    \Pro \left( |u(\tau_{1})|_{L^\infty(D)} = 3 |u(0)|_{L^\infty(D)} \text{ and } \tau_{1}<\e \ \Big| \  |u(0)| = 3^n c_0\right) \leq C \e^{q}.
  \end{equation}
  }
  The mild solution $u$ solves
  \begin{equation}
    u(t ) = S(t)u(0) + \int_{0}^{t} S(t-s) f(u(s))ds + \int_{0}^{ t} S(t-s) \sigma(u(s))dW(s).
  \end{equation}
  Define $Z(t) = \int_0^t S(t-s) \sigma(u(s))dW(s)$.

  By Lemma \ref{lem:decay}, if $\sup_{t \in [0,\e\wedge \tau_1]}|Z(t)|_{L^\infty(D)} \leq \frac{1}{3}\inf_{[0,\e \wedge \tau_1]} |u(t)|_{L^\infty(D)}$ and $|u(0)|_{L^\infty(D)} = 3^n c_0$, then for all $t \in [0,\e \wedge \tau_1]$,
  \begin{equation}
    |u(t)|_{L^\infty(D)} \leq \frac{3}{2}\left( |u(0)|_{L^\infty(D)}^{-(\beta-1)} + \frac{K_1}{2^\beta(\beta-1)} t \right)^{-\frac{1}{\beta -1}} \leq \frac{3}{2} |u(0)|_{L^\infty(D)},
  \end{equation}
  implying that $|u(t)|_{L^\infty(D)}$ cannot reach $3 |u(0)|_{L^\infty(D)}$ for $t< \e \wedge \tau_1$. Therefore,

  \begin{align}
    &\Pro\left( |u(\tau_1)|_{L^\infty(D)} = 3 |u(0)|_{L^\infty(D)} \text{ and } \tau_1<\e \ \Big| \  |u(0)| = 3^n c_0\right) \nonumber\\
    &\leq    \Pro \left(\sup_{t \in [0,\e\wedge \tau_1]} |Z(t)|_{L^\infty(D)} > \frac{1}{3} \inf_{t \in [0,\e \wedge \tau_1]} |u(t)|_{L^\infty(D)} \ \Big| \  |u(0)|_{L^\infty(D)} = 3^n c_0 \right) \nonumber\\
    & \leq \Pro \left(\sup_{t \in [0,\e\wedge \tau_1]} |Z(t)|_{L^\infty(D)} \geq 3^{n-2}c_0 \ \Big| \  |u(0)|_{L^\infty(D)} = 3^n c_0\right).
  \end{align}
  The last inequality in the above display is a consequence of the definition of $\tau_1$ in \eqref{eq:tau-def}. The value of $|u(t)|_{L^\infty(D)}$ cannot drop below $\frac{1}{3}|u(0)|_{L^\infty(D)} = 3^{n-1}c_0$ if $t \in [0, \tau_1]$.

  Define the stopping time $\tilde{\tau}_1 = \inf\{t>0: |Z(t)|_{L^\infty(D)} \geq 3^{n-2}c_0\}$ and notice that
  \begin{align}
    &\Pro \left(\sup_{t \in [0,\e\wedge \tau_1]} |Z(t)|_{L^\infty(D)} \geq 3^{n-2}c_0 \ \Big| \  |u(0)|_{L^\infty(D)} = 3^n c_0\right) \nonumber \\
    &= \Pro \left(\sup_{t \in [0,\e\wedge \tau_1 \wedge \tilde \tau_1]} |Z(t)|_{L^\infty(D)} \geq 3^{n-2}c_0 \ \Big| \  |u(0)|_{L^\infty(D)} = 3^n c_0\right).
  \end{align}

  By Chebyshev's inequality and \eqref{eq:stoch-conv-sup-norm}, for $\alpha, \zeta,$ and $p$ to be specified soon,
  \begin{align} \label{eq:z-bound-in-terms-of-z-alpha}
    &\Pro \left(\sup_{t \in [0,\e\wedge \tau_1 \wedge \tilde \tau_1]} |Z(t)|_{L^\infty(D)} \geq 3^{n-2}c_0 \ \Big| \  |u(0)|_{L^\infty(D)} = 3^n c_0\right) \nonumber \\
    & \displaystyle \leq \frac{\E \sup_{t \in [0,\e\wedge \tau_1 \wedge \tilde \tau_1]} |Z(t)|_{L^\infty(D)}^p}{3^{p(n-2)} c_0^p}\nonumber \\
    & \displaystyle \leq  \frac{C\e^{p(\alpha - \frac{\zeta}{2})-1} \int_0^\e \E |\tilde{Z}_\alpha(s)|_{L^p(D)}^pds}{3^{p(n-2)}c_0^p}
  \end{align}
  where
  \begin{equation}
    \tilde{Z}_\alpha(t) = \int_0^t (t-s)^{-\alpha}S(t-s) \sigma(u(s))\mathbbm{1}_{\{s \leq \tau_1 \wedge \tilde{\tau}_1\}}dW(s).
  \end{equation}

  By \eqref{eq:BDG-stoch-conv}, for $t \in [0,\e]$,
  \begin{equation}
    \E |\tilde{Z}_\alpha(t)|_{L^p(D)}^p
    \leq C \E \left(\int_0^{t} (t-s)^{-2\alpha-\eta} |\sigma(u(s))|_{L^\infty(D)}^2 \mathbbm{1}_{\{s \leq \tau_1 \wedge \tilde{\tau}_1\}}ds \right)^{\frac{p}{2}}.
  \end{equation}
  By \eqref{eq:sig-assum},
  \begin{equation}
    \E |\tilde{Z}_\alpha(t)|_{L^p(D)}^p
    \leq C \E\left(\int_0^{t} (t-s)^{-2\alpha-\eta} (1 +   |u(s)|_{L^\infty(D)}^{2\gamma})\mathbbm{1}_{\{s \leq \tau_1 \wedge \tilde{\tau}_1\}}ds \right)^{\frac{p}{2}}.
  \end{equation}
  By the definitions of $\tau_1$ and $\tilde \tau_1$, $|Z(t)|_{L^\infty(D)} \leq \frac{1}{3}|u(t)|_{L^\infty(D)}$ for $t \in[0, \tau_1 \wedge \tilde \tau_1]$. By Lemma \ref{lem:decay}, for $t \in [0,\e]$
  \begin{align}
    &\E |\tilde{Z}_\alpha(t)|_{L^p(D)}^p \nonumber \\
    &\leq C \E\left(\int_0^{t} (t-s)^{-2\alpha-\eta} \left(1 +   \left(|u(0)|^{-(\beta -1)} + Cs \right)^{-\frac{2\gamma}{\beta -1}}\right)ds \right)^{\frac{p}{2}}.
  \end{align}
  Now we make the observation that
  \begin{equation}
    \left(|u(0)|^{-(\beta -1)} + Cs \right)^{-\frac{2\gamma}{\beta -1}} \leq C|u(0)|^2 s^{- \frac{2(\gamma-1)}{\beta -1}}.
  \end{equation}
  leading to the estimate that
  \begin{equation} \label{eq:Z-alpha-bound-Beta}
    \E |\tilde{Z}_\alpha(t)|_{L^p(D)}^p
    \leq C \left(\int_0^t (t-s)^{-2\alpha-\eta} \left(1 +   |u(0)|^2s^{-\frac{(2(\gamma-1))}{\beta -1}}\right)ds \right)^{\frac{p}{2}}
  \end{equation}
  In Assumption \ref{assum:f-sig}, we assumed that $\gamma< 1 + \frac{(1-\eta)(\beta -1)}{2}$. Therefore, we can choose $\alpha \in (0,1)$ small enough so that
  \begin{equation}
    2\alpha + \eta + \frac{2(\gamma -1)}{(\beta -1)} = 1.
  \end{equation}
  Therefore, the integral $\int_0^t (t-s)^{-2 \alpha - \eta} s^{-\frac{2(\gamma-1)}{\beta-1}}ds = \frac{\pi}{\sin(\pi(2\alpha + \eta))}$  in \eqref{eq:Z-alpha-bound-Beta} is a Beta function whose value  does not depend on $t$. For $t \in [0,\e ]$,
  \begin{equation} \label{eq:Z-alpha-unif-bound}
    \E |\tilde{Z}_\alpha(t)|_{L^p(D)}^p
    \leq C (\e^{(1-2\alpha-\eta) \frac{p}{2}} + |u(0)|_{L^\infty(D)}^p).
  \end{equation}

  Now that $\alpha$ has been chosen, we choose $\zeta \in (0,2\alpha)$ and $p> \frac{1}{\alpha - \frac{\zeta}{2}}$ large enough so that $p \zeta > d$ (the spatial dimension). Then plugging this into \eqref{eq:z-bound-in-terms-of-z-alpha},
  \begin{align} \label{eq:power-bound}
    &\Pro \left(\sup_{t \in [0,\e\wedge \tau_1 \wedge \tilde \tau_1]} |Z(t)|_{L^\infty(D)} \geq 3^{n-2}c_0 \ \Big| \  |u(0)|_{L^\infty(D)} = 3^n c_0\right) \nonumber \\
    &\leq \frac{C \e^{p(\alpha - \zeta)-1}  \int_0^\e (\e^{(1-2\alpha - \eta)\frac{p}{2}}+ 3^{n p} c_0^p)ds}{3^{p(n-2)}c_0^p} \nonumber\\
    &\leq C \e^{p(\alpha - \frac{\zeta}{2})}.
  \end{align}
  We can set $q=p(\alpha - \frac{\zeta}{2})>1$ to finish the proof.
\end{proof}

The $n=1$ case was excluded from the previous lemma because it is slightly different and significantly easier to prove.

\begin{lemma} \label{lem:grow-quick-n=1}
  There exists $C>0$, $q>1$, and $\e_0>0$ such that  for any $\e \in (0,\e_0)$ and \textcolor{black}{$k \in \mathbb{N}$},
  \begin{equation}
    \Pro\left(\tau_{k+1} - \tau_k>\e \ \Big| \ |u(\tau_k)|_{L^\infty(D)} = 3c_0\right) \leq C \e^q.
  \end{equation}
\end{lemma}

\begin{proof}
  \textcolor{black}{As we argued at the beginning of the proof of Lemma \ref{lem:cond-prob-grow-quick}, we can assume without loss of generality that $k=0$. It suffices to prove that
  \begin{equation}
    \Pro\left(\tau_{1} >\e \ \Big| \ |u(0)|_{L^\infty(D)} = 3c_0\right) \leq C \e^q.
  \end{equation}}
  If $|u(0)| = 3c_0$, then the next step of the Markov chain must go up to $|u(\tau_1)|_{L^\infty(D)} = 3^2 c_0$. See \eqref{eq:tau-def}.

  For $t \in [0,\tau_1]$, $|u(t)|_{L^\infty(D)}\leq 3^2c_0$. Because $f$ and $\sigma$ are both continuous, there exists $K>0$ such that
  \[|f(u(t))|_{L^\infty(D)}\leq K \text{ and } |\sigma(u(t))|_{L^\infty(D)} \leq K \text{ for } t \in [0,\tau_1].\]

  Let
  \begin{equation}
    Z(t) = \int_0^t S(t-s)\sigma(u(s))dw(s)
  \end{equation}
  so that
  by the definition of the mild solution \eqref{eq:mild}, if $|u(0)|_{L^\infty(D)} = 3 c_0$, then for $\e>0$,
  \begin{align}
    &\sup_{t \in [0,\e \wedge \tau_1]}|u(t)|_{L^\infty(D)} \nonumber \\
    &\leq |u(0)|_{L^\infty(D)} + \sup_{t \in [0,\e\wedge \tau_1]}\left|\int_0^t S(t-s) f(u(s))ds \right|_{L^\infty(D)} + \sup_{t \in [0,\e\wedge \tau_1]}\left|Z(t) \right|_{L^\infty(D)}\nonumber \\
    &\leq 3c_0 + K\e + \sup_{t \in [0,\e\wedge \tau_1]}\left|Z(t) \right|_{L^\infty(D)}.
  \end{align}

  Therefore, if $\sup_{t \in [0,\e \wedge \tau_1]}|u(t)|_{L^\infty(D)}>3^2c_0$, then
  \begin{equation}
    \sup_{t \in [0,\e\wedge \tau_1]}\left|Z(t) \right|_{L^\infty(D)}>3^2 c_0 - 3c_0 - K\e.
  \end{equation}

  By the factorization method \eqref{eq:BDG-stoch-conv}--\eqref{eq:stoch-conv-sup-norm}, choosing $p,\zeta, \alpha$ to have the same values as in \eqref{eq:power-bound}
  \begin{equation}
    \E \sup_{t \in [0,\e\wedge \tau_1]}\left|Z(t) \right|_{L^\infty(D)}^p \leq C \e^{p(\alpha - \frac{\zeta}{2})}K^{ p}.
  \end{equation}

  By Chebyshev's inequality,
  \begin{align}
    &\Pro \left(\tau_1<\e \ \Big| \ |u(0)|_{L^\infty(D)} = 3c_0 \right) \nonumber\\
    &\leq \Pro \left(\sup_{t \in [0,\e\wedge \tau_1]}\left|Z(t) \right|_{L^\infty(D)}>3^2 c_0 - 3c_0 - K\e \ \Big| \ |u(0)|_{L^\infty(D)} = 3c_0 \right) \nonumber\\
    &\leq \frac{C \e^{p(\alpha - \frac{\zeta}{2})}K^{p}}{\left(3^2c_0 - 3c_0 - K\e \right)^p}.
  \end{align}

\end{proof}

Now we prove Theorem \ref{thm:global-existence}.

\begin{proof}[Proof of Theorem \ref{thm:global-existence}]
  Let $\tau_k$ be defined as in \eqref{eq:tau-def}. \textcolor{black}{By Lemmas \ref{lem:cond-prob-grow-quick}-- \ref{lem:grow-quick-n=1} }, there exists $C>0$ and $q>1$ such that for any $k \in \mathbb{N}$, and small $\e>0$
  \begin{align}
    &\Pro(\tau_{k+1} - \tau_k < \e \text{ and } |u(\tau_{k+1})|_{L^\infty(D)} = 3|u(\tau_k)|_{L^\infty(D)})
    \leq C \e^q.
  \end{align}
  In particular, for any $k \in \mathbb{N}$,
  \begin{equation}
    \Pro\left(\tau_{k+1} - \tau_k < \frac{1}{k} \text{ and } |u(\tau_{k+1})|_{L^\infty(D)} = 3|u(\tau_k)|_{L^\infty(D)}\right) \leq \frac{C }{k^q}.
  \end{equation}
  Because $q>1$,
  \begin{equation} \label{eq:BC-lemma}
    \sum_{k=1}^\infty \Pro\left(\tau_{k+1} - \tau_k < \frac{1}{k} \text{ and } |u(\tau_{k+1})|_{L^\infty(D)} = 3|u(\tau_k)|_{L^\infty(D)}\right)< +\infty.
  \end{equation}
  By the Borel-Cantelli Lemma, with probability one there exists a (random) index $N_0(\omega)>0$ such that for all $k \geq N_0(\omega)$, either
  \begin{equation}
    \tau_{k+1} - \tau_k \geq \frac{1}{k} \ \ \ \text{ or } \ \ \ |u(\tau_{k+1})|_{L^\infty(D)} = \frac{1}{3}|u(\tau_k)|_{L^\infty(D)}.
  \end{equation}
  Either the $|u(\tau_k)|_{L^\infty(D)}$ decreases, or the time required increase is greater than $\frac{1}{k}$.

  From the definition of $\tau_k$, $\min_{k \geq N_0(\omega)} |u(\tau_k)|_{L^\infty(D)}$ is attained. We can choose $N_1(\omega)> N_0(\omega)$ such that for $k \geq N_1(\omega)$, $|u(\tau_k)|_{L^\infty(D)} \geq |u(\tau_{N_1(\omega)})|_{L^\infty(D)}$.

  For any $m \geq N_1(\omega)> N_0(\omega)$,
  \begin{equation} \label{eq:sum-of-tau-incr}
    \sum_{k=N_1(\omega)}^m (\tau_{k+1} - \tau_k) \geq \sum_{k=N_1(\omega)}^m \frac{1}{k} \mathbbm{1}_{\{|u(\tau_{k+1})|_{L^\infty(D)}=3 |u(\tau_{k})|_{L^\infty(D)} \}}.
  \end{equation}

  Because of the definition of $N_1$,  there must always be more steps where $|u(\tau_k)|_{L^\infty(D)}$ increases than steps where it decreases. Therefore, for any $m \geq N_1(\omega)$
  \begin{equation} \label{eq:half-steps-up}
     U_m(\omega) := \sum_{N_1(\omega)}^{m-1} \mathbbm{1}_{\{|u(\tau_{k+1})|_{L^\infty(D)} = 3 |u(\tau_{k})|_{L^\infty(D)} \}} \geq \frac{(m-N_1(\omega))}{2}
  \end{equation}

  By the summation by parts formula and \eqref{eq:half-steps-up}
  \begin{align}
    &\sum_{N_1(\omega)}^m \frac{1}{k} \mathbbm{1}_{\{|u(\tau_{k+1})|_{L^\infty(D)} = 3 |u(\tau_{k})|_{L^\infty(D)} \}} \nonumber\\
    & = \sum_{N_1(\omega)}^m \frac{1}{k} (U_{k+1}(\omega) - U_{k}(\omega)) \nonumber\\
    & = \frac{U_{m+1}(\omega)}{m} - \sum_{k=N_1(\omega) + 1}^m U_k(\omega)\left(\frac{1}{k} - \frac{1}{k-1} \right)\nonumber\\
    &= \frac{U_m(\omega)}{m}  + \sum_{k=N_1(\omega) +1} U_k(\omega) \left( \frac{1}{k(k-1)} \right) \nonumber\\
    &\geq \frac{m-N_1(\omega)}{2m} + \sum_{k=N_1(\omega) + 1}^m \frac{(k - N_1(\omega))}{2k(k-1)}.
  \end{align}
  This sum diverges as $m \to +\infty$.

  Therefore, by \eqref{eq:sum-of-tau-incr},
  \begin{equation}
    \sum_{k=N_1(\omega)}^\infty (\tau_{k+1} - \tau_k) =+\infty \text{ with probability one}
  \end{equation}
  and the solutions cannot explode in finite time.
\end{proof}

\section{Comparison with Mueller's result \cite{m-2000}} \label{S:comparison}
Consider the case of a stochastic heat equation on a one-dimensional interval domain $D=[0,\pi]$ exposed to a space-time white noise and a polynomially dissipative forcing
\begin{equation} \label{eq:one-dim-heat}
  \begin{cases}
    \frac{\partial u}{\partial t}(t,x) = \Delta u(t,x) - |u(t,x)|^{\beta-1}u(t,x) + (1 + |u(t,x)|^\gamma) \dot{W}(t,x),\\
    u(t,0) = u(t,\pi) = 0,\\
    u(0,x) = u_0(x).
  \end{cases}
\end{equation}

The eigenvalues of the $\frac{\partial^2}{\partial x^2}$ operator are $-\alpha_k = -k^2$, for $k \in \mathbb{N}$ and the eigenfunctions are $e_k(x) = \sqrt{\frac{2}{\pi}}\sin(kx)$ are uniformly bounded. We take $\dot{W}$ to be space-time white noise. In the language of Assumption \ref{assum:W}, this means that $\lambda_j \equiv 1$. This satisfies the assumptions of \eqref{eq:eta-def} for any $\beta>\frac{1}{2}$, $\rho=+\infty$ and $\eta = \theta> \frac{1}{2}$.

Theorem \ref{thm:global-existence} proves that mild solutions to \eqref{eq:one-dim-heat} never explode for all $\gamma, \beta$ satisfying $\beta>1$ and
\begin{equation}
  \gamma < 1 + \frac{\beta -1}{4} = \frac{3 + \beta}{4}.
\end{equation}

Mueller \cite{m-2000} proved that when $f \equiv 0$, solutions can explode in finite time whenever $\gamma > \frac{3}{2}$. Theorem \ref{thm:global-existence} proves that adding sufficiently strong dissipative forcing to the equation prevents explosion for arbitrarily large $\gamma$.

Mueller \cite{m-2000} also proved that when $f \equiv 0$ and $\gamma< \frac{3}{2}$, that solutions never explode. This is due to the dissipative effects of the Laplace operator. Theorem \ref{thm:global-existence} focuses on the role that the dissipativity of $f$ plays, but ignores the dissipation due to $\Delta$.
A stronger result is possible by combining the results of \cite{m-2000} and Theorem \ref{thm:global-existence}. to conclude that any mild solution to \eqref{eq:one-dim-heat} will be global in time for all
\begin{equation}
  \gamma < \max\left\{ \frac{3}{2}, \frac{3 + \beta}{4} \right\}.
\end{equation}
Mueller's result dominates when $\beta\leq 3$ and Theorem \ref{thm:global-existence} dominates when $\beta>3$.

\begin{appendix}
  \section{Moment bounds for the stochastic convolution}
    In this appendix we prove Propositions \ref{prop:Z-alpha}--\ref{prop:Z-sup-bound}.

    \begin{proof}[Proof of Proposition \ref{prop:Z-alpha}]
      This is similar to Lemma 4.1 of \cite{c-2009}. Let $K(t,x,y) = \sum_{k=1}^\infty e^{-\alpha_k t} e_k(x)e_k(y)$ be the kernel of the semigroup for $t \geq 0$, $x,y \in D$. In this way, for $x \in D$, from Assumption \ref{assum:W}
      \begin{align}
        \tilde{Z}_\alpha(t,x) &= \int_0^t \int_D (t-s)^{-\alpha}K(t-s,x,y) \sigma(u(s,y))\mathbbm{1}_{\{s \leq \tau\}}W(dyds) \nonumber\\
        &= \sum_{j=1}^\infty \int_0^t \int_D (t-s)^{-\alpha} K(t-s,x-y) \sigma(u(s,y)) \mathbbm{1}_{\{s \leq \tau\}}\lambda_j e_j(y)dydB_j(s).
      \end{align}
      For any fixed $x$, this is a real-valued stochastic integral. By the BDG inequality,
      \begin{align}
        &\E|\tilde{Z}_\alpha(t,x)|^p \nonumber\\
        &\leq C_p \E\left(\sum_{j=1}^\infty \int_0^t (t-s)^{-2\alpha}\mathbbm{1}_{\{s \leq \tau\}}\left(\int_D K(t-s,x,y) \sigma(u(s,y))e_j(y) dy\right)^2\lambda_j^2 ds \right)^{\frac{p}{2}}
      \end{align}
      Apply a H\"older inequality to the infinite sum
      \begin{align}
        &\sum_{j=1}^\infty \lambda_j^2 \left(\int_D K(t-s,x,y) \sigma(u(s,y))e_j(y) dy\right)^2 \nonumber \\
         &\leq \left(\sum_{j=1}^\infty \lambda_j^\rho \left(\int_D K(t-s,x,y) \sigma(u(s,y))e_j(y) dy\right)^2   \right)^{\frac{2}{\rho}} \nonumber\\
         &\qquad\times \left(\sum_{j=1}^\infty \left(\int_D K(t-s,x,y) \sigma(u(s,y))e_j(y) dy\right)^2 \right)^{\frac{\rho-2}{\rho}}.
      \end{align}

      Because $K$ is the kernel of a contraction semigroup, for any $j \in \mathbb{N}$,
      \begin{equation}
        \left(\int_D K(t-s,x,y) \sigma(u(s,y))e_j(y) dy\right)^2 \leq |\sigma(u(s))|_{L^\infty(D)}^2 |e_j|_{L^\infty(D)}^2.
      \end{equation}
      Applying this estimate to the first term of the product
      \begin{align}
        &\sum_{j=1}^\infty \left(\int_D K(t-s,x,y) \sigma(u(s,y))e_j(y)\lambda_j dy\right)^2  \nonumber\\
        &\leq \left(\sum_{j=1}^\infty \lambda_j^\rho |\sigma(u(s))|_{L^\infty(D)}^2|e_j|_{L^\infty(D)}^2   \right)^{\frac{2}{\rho}} \nonumber \\ &\qquad\times\left(\sum_{j=1}^\infty \left(\int_D K(t-s,x,y) \sigma(u(s,y))e_j(y) dy\right)^2 \right)^{\frac{\rho-2}{\rho}}.
      \end{align}
      On the other hand, because $\{e_j\}$ is a complete orthonormal basis of $L^2(D)$,
      \begin{align}
        &\sum_{j=1}^\infty \left(\int_D K(t-s,x,y) \sigma(u(s,y))e_j(y) dy\right)^2 \nonumber\\
        &= \int_D |K(t-s,x,y)|^2 |\sigma(u(s,y))|^2dy \nonumber\\
        &\leq |\sigma(u(s))|_{L^\infty(D)}^2 \int_D |K(t-s,x,y)|^2 dy.
      \end{align}
      Using the fact that $\{e_k\}$ is a complete orthonormal basis of $L^2(D)$,
      \begin{align}
        &\int_D (K(t-s,x,y))^2 dy = \int_D \left(\sum_{k=1}^\infty e^{-\alpha_k (t-s)}e_k(x)e_k(y) \right)^2dy \nonumber\\
        &= \sum_{k=1}^\infty e^{-2\alpha_k (t-s)} |e_k(x)|^2.
      \end{align}
      By Assumption \eqref{eq:alpha-sum} and the fact that $\sup_{u>0} u^{\theta} e^{-u} =:c_\theta< +\infty$,
      \begin{align}
        &\int_D (K(t-s,x,y))^2 dy \leq \sum_{k=1}^\infty e^{-2\alpha_k (t-s)} |e_k(x)|^2 \nonumber\\
        &\leq \sum_{k=1}^\infty \frac{(2\alpha_k (t-s))^{\theta}}{(2\alpha_k (t-s))^\theta} e^{-2 \alpha_k (t-s)} |e_k|_{L^\infty(D)}^2 \leq C(t-s)^{-\theta}.
      \end{align}
      Combining all of these estimates and using the fact that $\frac{\theta(\rho-2)}{\rho} = \eta$, we conclude that for any fixed $x \in D$,
      \begin{equation}
        \E|\tilde{Z}_\alpha(t,x)|^p \leq C_{p,\alpha} \E \left(\int_0^t |\sigma(u(s))|_{L^\infty(D)}^2 (t-s)^{-\eta - 2\alpha} \mathbbm{1}_{\{s\leq \tau\}}ds \right)^{\frac{p}{2}}
      \end{equation}
      This estimate is uniform with respect to $x \in D$ so if we integrate over all $x \in D$,
      \begin{equation}
        \E|\tilde{Z}_\alpha(t)|^p_{L^p(D)} \leq C_{p,\alpha} \E \left(\int_0^t |\sigma(u(s))|_{L^\infty(D)}^2 \mathbbm{1}_{\{s\leq \tau\}} (t-s)^{-\eta - 2\alpha} ds \right)^{\frac{p}{2}}.
      \end{equation}
      with a larger constant.
    \end{proof}

    Before we present the proof of this Proposition \ref{prop:Z-sup-bound}, we introduce the fractional Sobolev spaces $ W^{\zeta,p}(D)$ for $\zeta \in (0,1)$ and $p\geq 1$. The $W^{\zeta,p}$ space is endowed with the norm
      \begin{equation}
        |\varphi|_{W^{\zeta,p}(D)}^p = |\varphi|_{L^p(D)}^p + \int_D \int_D \frac{|\varphi(x) - \varphi(y)|^p}{|x-y|^{d+\zeta p}}dxdy.
      \end{equation}
      In the above expression $d$ is the spatial dimension of $D$. We use two important facts about these fractional Sobolev spaces. The fractional Sobolev embedding theorem \cite[Theorem 8.2]{sobolev} implies that when $\zeta p >d$, $W^{\zeta,p}$ embeds continuously into the H\"older space $C^\vartheta(D)$ with $\vartheta = \frac{\zeta p - d}{p}$. There exists a constant $C_{\zeta,p}$ such that for all $\varphi \in W^{\zeta,p}$,
      \begin{equation} \label{eq:Sobolev-embedding}
        |\varphi|_{C^\vartheta} \leq C_{\zeta,p} |\varphi|_{W^{\zeta,p}}.
      \end{equation}
      Furthermore, regularizing properties of elliptic semigroups imply that there exists $C_{\zeta,p}>0$ such that for any $t>0$, and $\varphi \in L^p(D)$,
      \begin{equation} \label{eq:semigroup-sobolev}
        |S(t)\varphi|_{W^{\zeta,p}(D)} \leq C t^{-\frac{\zeta}{2}}|\varphi|_{L^p(D)}
      \end{equation}

    \begin{proof}[Proof of Proposition \ref{prop:Z-sup-bound}]
      Let $\alpha \in \left( 0, \frac{1-\eta}{2}\right)$, $\zeta \in (0,2\alpha)$. By the factorization lemma,
      \begin{equation}
        Z(t \wedge \tau) = \frac{\sin(\pi \alpha)}{\pi}\int_0^{t \wedge \tau}(t\wedge \tau -s)^{\alpha-1} S(t \wedge \tau-s) \tilde{Z}_\alpha(s)ds.
      \end{equation}
      By Propositon 5.9 of \cite{dpz-book} along with \eqref{eq:Sobolev-embedding}--\eqref{eq:semigroup-sobolev}, $(t,x) \mapsto Z(t \wedge \tau, x)$ is almost surely continuous.   By \eqref{eq:Sobolev-embedding}--\eqref{eq:semigroup-sobolev} and H\"older's inequality, for $p> \max\left\{\frac{d}{\zeta}, \frac{1}{\alpha-\frac{\zeta}{2}}\right\}$,
      \begin{align}
        &\E \sup_{t \in [0,T]} \sup_{x \in D} |Z(t\wedge \tau,x)|^p \nonumber\\
        &\leq C \sup_{t \in [0,T]} \left|\int_0^{t \wedge \tau} (t\wedge \tau-s)^{\alpha -1 - \frac{\zeta}{2}}|\tilde{Z}_\alpha(s)|_{L^p(D)}^pds \right|^p \nonumber\\
        &\leq C\left(\int_0^T s^{\frac{(\alpha-1-\frac{\zeta}{2})p}{p-1}}ds \right)^{p-1} \int_0^T \E|\tilde{Z}_\alpha(s)|_{L^p(D)}^pds \nonumber\\
        &\leq C_{\alpha,\zeta, p, T} T^{p(\alpha - \frac{\zeta}{2})-1} \int_0^T \E |\tilde{Z}_\alpha(s)|^p_{L^p(D)}ds.
      \end{align}

    \end{proof}

\end{appendix}
\bibliographystyle{amsplain}
\bibliography{strong-dissip}

\begin{bibdiv}
\begin{biblist}

\bib{b-2018}{article}{
      author={Bezdek, Pavel},
       title={Existence and blow-up of solutions to the fractional stochastic
  heat equations},
        date={2018},
        ISSN={2194-0401},
     journal={Stoch. Partial Differ. Equ. Anal. Comput.},
      volume={6},
      number={1},
       pages={73\ndash 108},
         url={https://doi-org.ezproxy.bu.edu/10.1007/s40072-017-0103-8},
      review={\MR{3768995}},
}

\bib{bp-1999}{article}{
      author={Brze\'{z}niak, Zdzis{\l}aw},
      author={Peszat, Szymon},
       title={Space-time continuous solutions to {SPDE}'s driven by a
  homogeneous {W}iener process},
        date={1999},
        ISSN={0039-3223},
     journal={Studia Math.},
      volume={137},
      number={3},
       pages={261\ndash 299},
         url={https://doi-org.ezproxy.bu.edu/10.4064/sm-137-3-261-299},
      review={\MR{1736012}},
}

\bib{cerrai-book}{book}{
      author={Cerrai, Sandra},
       title={Second order {PDE}'s in finite and infinite dimension},
      series={Lecture Notes in Mathematics},
   publisher={Springer-Verlag, Berlin},
        date={2001},
      volume={1762},
        ISBN={3-540-42136-X},
         url={https://doi-org.ezproxy.bu.edu/10.1007/b80743},
        note={A probabilistic approach},
      review={\MR{1840644}},
}

\bib{c-2003}{article}{
      author={Cerrai, Sandra},
       title={Stochastic reaction-diffusion systems with multiplicative noise
  and non-{L}ipschitz reaction term},
        date={2003},
        ISSN={0178-8051},
     journal={Probab. Theory Related Fields},
      volume={125},
      number={2},
       pages={271\ndash 304},
         url={https://doi-org.ezproxy.bu.edu/10.1007/s00440-002-0230-6},
      review={\MR{1961346}},
}

\bib{c-2009}{article}{
      author={Cerrai, Sandra},
       title={A {K}hasminskii type averaging principle for stochastic
  reaction-diffusion equations},
        date={2009},
        ISSN={1050-5164},
     journal={Ann. Appl. Probab.},
      volume={19},
      number={3},
       pages={899\ndash 948},
         url={https://doi-org.ezproxy.bu.edu/10.1214/08-AAP560},
      review={\MR{2537194}},
}

\bib{c-2011}{article}{
      author={Cerrai, Sandra},
       title={Averaging principle for systems of reaction-diffusion equations
  with polynomial nonlinearities perturbed by multiplicative noise},
        date={2011},
        ISSN={0036-1410},
     journal={SIAM J. Math. Anal.},
      volume={43},
      number={6},
       pages={2482\ndash 2518},
         url={https://doi-org.ezproxy.bu.edu/10.1137/100806710},
      review={\MR{2854919}},
}

\bib{cdf-2013}{article}{
      author={Cerrai, Sandra},
      author={Da~Prato, Giuseppe},
      author={Flandoli, Franco},
       title={Pathwise uniqueness for stochastic reaction-diffusion equations
  in {B}anach spaces with an {H}\"{o}lder drift component},
        date={2013},
        ISSN={2194-0401},
     journal={Stoch. Partial Differ. Equ. Anal. Comput.},
      volume={1},
      number={3},
       pages={507\ndash 551},
         url={https://doi-org.ezproxy.bu.edu/10.1007/s40072-013-0016-0},
      review={\MR{3327515}},
}

\bib{chow-2009}{article}{
      author={Chow, Pao-Liu},
       title={Unbounded positive solutions of nonlinear parabolic {I}t\^{o}
  equations},
        date={2009},
     journal={Commun. Stoch. Anal.},
      volume={3},
      number={2},
       pages={211\ndash 222},
      review={\MR{2548104}},
}

\bib{chow-2011}{article}{
      author={Chow, Pao-Liu},
       title={Explosive solutions of stochastic reaction-diffusion equations in
  mean {$L^p$}-norm},
        date={2011},
        ISSN={0022-0396},
     journal={J. Differential Equations},
      volume={250},
      number={5},
       pages={2567\ndash 2580},
         url={https://doi-org.ezproxy.bu.edu/10.1016/j.jde.2010.11.008},
      review={\MR{2756076}},
}

\bib{dr-2002}{article}{
      author={Da~Prato, Giuseppe},
      author={R\"{o}ckner, Michael},
       title={Singular dissipative stochastic equations in {H}ilbert spaces},
        date={2002},
        ISSN={0178-8051},
     journal={Probab. Theory Related Fields},
      volume={124},
      number={2},
       pages={261\ndash 303},
         url={https://doi-org.ezproxy.bu.edu/10.1007/s004400200214},
      review={\MR{1936019}},
}

\bib{dpz-book}{book}{
      author={Da~Prato, Giuseppe},
      author={Zabczyk, Jerzy},
       title={Stochastic equations in infinite dimensions},
     edition={Second},
      series={Encyclopedia of Mathematics and its Applications},
   publisher={Cambridge University Press, Cambridge},
        date={2014},
      volume={152},
        ISBN={978-1-107-05584-1},
         url={https://doi-org.ezproxy.bu.edu/10.1017/CBO9781107295513},
      review={\MR{3236753}},
}

\bib{dkmnx-book}{book}{
      author={Dalang, Robert},
      author={Khoshnevisan, Davar},
      author={Mueller, Carl},
      author={Nualart, David},
      author={Xiao, Yimin},
       title={A minicourse on stochastic partial differential equations},
      series={Lecture Notes in Mathematics},
   publisher={Springer-Verlag, Berlin},
        date={2009},
      volume={1962},
        ISBN={978-3-540-85993-2},
        note={Held at the University of Utah, Salt Lake City, UT, May 8--19,
  2006, Edited by Khoshnevisan and Firas Rassoul-Agha},
      review={\MR{1500166}},
}

\bib{d-1999}{article}{
      author={Dalang, Robert~C.},
       title={Extending the martingale measure stochastic integral with
  applications to spatially homogeneous s.p.d.e.'s},
        date={1999},
        ISSN={1083-6489},
     journal={Electron. J. Probab.},
      volume={4},
       pages={no. 6, 29},
         url={https://doi-org.ezproxy.bu.edu/10.1214/EJP.v4-43},
      review={\MR{1684157}},
}

\bib{sobolev}{article}{
      author={Di~Nezza, Eleonora},
      author={Palatucci, Giampiero},
      author={Valdinoci, Enrico},
       title={Hitchhiker's guide to the fractional {S}obolev spaces},
        date={2012},
        ISSN={0007-4497},
     journal={Bull. Sci. Math.},
      volume={136},
      number={5},
       pages={521\ndash 573},
         url={https://doi-org.ezproxy.bu.edu/10.1016/j.bulsci.2011.12.004},
      review={\MR{2944369}},
}

\bib{evans-book}{book}{
      author={Evans, Lawrence~C.},
       title={Partial differential equations},
      series={Graduate Studies in Mathematics},
   publisher={American Mathematical Society, Providence, RI},
        date={1998},
      volume={19},
        ISBN={0-8218-0772-2},
      review={\MR{1625845}},
}

\bib{fln-2019}{article}{
      author={Foondun, Mohammud},
      author={Liu, Wei},
      author={Nane, Erkan},
       title={Some non-existence results for a class of stochastic partial
  differential equations},
        date={2019},
        ISSN={0022-0396},
     journal={J. Differential Equations},
      volume={266},
      number={5},
       pages={2575\ndash 2596},
         url={https://doi-org.ezproxy.bu.edu/10.1016/j.jde.2018.08.039},
      review={\MR{3906260}},
}

\bib{grt-2020}{article}{
      author={Gordina, Maria},
      author={R\"{o}ckner, Michael},
      author={Teplyaev, Alexander},
       title={Ornstein-{U}hlenbeck processes with singular drifts: integral
  estimates and {G}irsanov densities},
        date={2020},
        ISSN={0178-8051},
     journal={Probab. Theory Related Fields},
      volume={178},
      number={3-4},
       pages={861\ndash 891},
         url={https://doi-org.ezproxy.bu.edu/10.1007/s00440-020-00991-w},
      review={\MR{4168390}},
}

\bib{i-1987}{article}{
      author={Iwata, Koichiro},
       title={An infinite-dimensional stochastic differential equation with
  state space {$C({\bf R})$}},
        date={1987},
        ISSN={0178-8051},
     journal={Probab. Theory Related Fields},
      volume={74},
      number={1},
       pages={141\ndash 159},
         url={https://doi-org.ezproxy.bu.edu/10.1007/BF01845644},
      review={\MR{863723}},
}

\bib{k-1992}{article}{
      author={Kotelenez, Peter},
       title={Existence, uniqueness and smoothness for a class of function
  valued stochastic partial differential equations},
        date={1992},
        ISSN={1045-1129},
     journal={Stochastics Stochastics Rep.},
      volume={41},
      number={3},
       pages={177\ndash 199},
         url={https://doi-org.ezproxy.bu.edu/10.1080/17442509208833801},
      review={\MR{1275582}},
}

\bib{wr-book}{book}{
      author={Liu, Wei},
      author={R\"{o}ckner, Michael},
       title={Stochastic partial differential equations: an introduction},
      series={Universitext},
   publisher={Springer, Cham},
        date={2015},
        ISBN={978-3-319-22353-7; 978-3-319-22354-4},
         url={https://doi-org.ezproxy.bu.edu/10.1007/978-3-319-22354-4},
      review={\MR{3410409}},
}

\bib{mz-1999}{article}{
      author={Manthey, Ralf},
      author={Zausinger, Thomas},
       title={Stochastic evolution equations in {$L^{2\nu}_\rho$}},
        date={1999},
        ISSN={1045-1129},
     journal={Stochastics Stochastics Rep.},
      volume={66},
      number={1-2},
       pages={37\ndash 85},
         url={https://doi-org.ezproxy.bu.edu/10.1080/17442509908834186},
      review={\MR{1687799}},
}

\bib{m-2018}{article}{
      author={Marinelli, Carlo},
       title={On well-posedness of semilinear stochastic evolution equations on
  {$L_p$} spaces},
        date={2018},
        ISSN={0036-1410},
     journal={SIAM J. Math. Anal.},
      volume={50},
      number={2},
       pages={2111\ndash 2143},
         url={https://doi-org.ezproxy.bu.edu/10.1137/16M108001X},
      review={\MR{3784905}},
}

\bib{mr-2010}{article}{
      author={Marinelli, Carlo},
      author={R\"{o}ckner, Michael},
       title={On uniqueness of mild solutions for dissipative stochastic
  evolution equations},
        date={2010},
        ISSN={0219-0257},
     journal={Infin. Dimens. Anal. Quantum Probab. Relat. Top.},
      volume={13},
      number={3},
       pages={363\ndash 376},
         url={https://doi-org.ezproxy.bu.edu/10.1142/S0219025710004152},
      review={\MR{2729590}},
}

\bib{m-1991}{article}{
      author={Mueller, Carl},
       title={Long time existence for the heat equation with a noise term},
        date={1991},
        ISSN={0178-8051},
     journal={Probab. Theory Related Fields},
      volume={90},
      number={4},
       pages={505\ndash 517},
         url={https://doi-org.ezproxy.bu.edu/10.1007/BF01192141},
      review={\MR{1135557}},
}

\bib{m-1998}{article}{
      author={Mueller, Carl},
       title={Long-time existence for signed solutions of the heat equation
  with a noise term},
        date={1998},
        ISSN={0178-8051},
     journal={Probab. Theory Related Fields},
      volume={110},
      number={1},
       pages={51\ndash 68},
         url={https://doi-org.ezproxy.bu.edu/10.1007/s004400050144},
      review={\MR{1602036}},
}

\bib{m-2000}{article}{
      author={Mueller, Carl},
       title={The critical parameter for the heat equation with a noise term to
  blow up in finite time},
        date={2000},
        ISSN={0091-1798},
     journal={Ann. Probab.},
      volume={28},
      number={4},
       pages={1735\ndash 1746},
         url={https://doi-org.ezproxy.bu.edu/10.1214/aop/1019160505},
      review={\MR{1813841}},
}

\bib{ms-1993}{article}{
      author={Mueller, Carl},
      author={Sowers, Richard},
       title={Blowup for the heat equation with a noise term},
        date={1993},
        ISSN={0178-8051},
     journal={Probab. Theory Related Fields},
      volume={97},
      number={3},
       pages={287\ndash 320},
         url={https://doi-org.ezproxy.bu.edu/10.1007/BF01195068},
      review={\MR{1245247}},
}

\bib{pz-2000}{article}{
      author={Peszat, Szymon},
      author={Zabczyk, Jerzy},
       title={Nonlinear stochastic wave and heat equations},
        date={2000},
        ISSN={0178-8051},
     journal={Probab. Theory Related Fields},
      volume={116},
      number={3},
       pages={421\ndash 443},
         url={https://doi-org.ezproxy.bu.edu/10.1007/s004400050257},
      review={\MR{1749283}},
}

\bib{ss-2000}{incollection}{
      author={Sanz-Sol\'{e}, Marta},
      author={Sarr\`a, M\`onica},
       title={Path properties of a class of {G}aussian processes with
  applications to spde's},
        date={2000},
   booktitle={Stochastic processes, physics and geometry: new interplays, {I}
  ({L}eipzig, 1999)},
      series={CMS Conf. Proc.},
      volume={28},
   publisher={Amer. Math. Soc., Providence, RI},
       pages={303\ndash 316},
         url={https://doi-org.ezproxy.bu.edu/10.1016/s0304-4149(98)00092-1},
      review={\MR{1803395}},
}

\bib{s-1992}{article}{
      author={Sowers, Richard~B.},
       title={Large deviations for a reaction-diffusion equation with
  non-{G}aussian perturbations},
        date={1992},
        ISSN={0091-1798},
     journal={Ann. Probab.},
      volume={20},
      number={1},
       pages={504\ndash 537},
  url={http://links.jstor.org.ezproxy.bu.edu/sici?sici=0091-1798(199201)20:1<504:LDFARE>2.0.CO;2-W&origin=MSN},
      review={\MR{1143433}},
}

\bib{w-1986}{incollection}{
      author={Walsh, John~B.},
       title={An introduction to stochastic partial differential equations},
        date={1986},
   booktitle={\'{E}cole d'\'{e}t\'{e} de probabilit\'{e}s de {S}aint-{F}lour,
  {XIV}---1984},
      series={Lecture Notes in Math.},
      volume={1180},
   publisher={Springer, Berlin},
       pages={265\ndash 439},
         url={https://doi-org.ezproxy.bu.edu/10.1007/BFb0074920},
      review={\MR{876085}},
}

\end{biblist}
\end{bibdiv}
\end{document}